\documentclass[a4paper,oneside,11pt]{article}

\usepackage{amsmath,amsfonts,amscd,amssymb}
\usepackage{longtable,geometry}
\usepackage[english]{babel}
\usepackage[utf8]{inputenc}
\usepackage[active]{srcltx}
\usepackage[T1]{fontenc}
\usepackage{graphicx}
\usepackage{enumitem}
\usepackage{bbm}
\usepackage{enumitem}   
\usepackage{mathtools}
\usepackage{hyperref}\hypersetup{colorlinks=false, pdfborder={0 0 0},}
\usepackage{mathrsfs} 

\usepackage{titling} 
\predate{} 
\postdate{} 

\usepackage{MnSymbol}
\usepackage{stmaryrd}
\usepackage{nicefrac}

 \usepackage{rotating}

\usepackage{xcolor}
\usepackage{framed}

\counterwithout{equation}{section}

\colorlet{shadecolor}{blue!15}

\newtheorem{theorem}{Theorem}[section]

\newtheorem{corollary}[theorem]{Corollary}
\newtheorem{lemma}[theorem]{Lemma}
\newtheorem{proposition}[theorem]{Proposition}

\newtheorem{remark}[theorem]{Remark}

\newcommand{\be}[1]{\begin{equation}\label{#1}}
\newcommand{\ee}{\end{equation}}

\newcommand{\ba}[1]{\begin{align}\label{#1}}
\newcommand{\ea}{\end{align}}

\newcommand{\ben}{\begin{equation*}}
\newcommand{\een}{\end{equation*}}

\newenvironment{proof}[1][\relax]
  {\paragraph{Proof\ifx#1\relax\else~of #1\fi}}%
  {~\hfill$\square$\par\bigskip}

\newcommand{\calM}{\mathcal{M}}

\newcommand{\calZ}{\mathcal{Z}}

\newcommand{\frX}{\mathfrak{X}}

\newcommand{\bbC}{\mathbb{C}}

\newcommand{\bbH}{\mathbb{H}}

\newcommand{\bbP}{\mathbb{P}}

\newcommand{\bbR}{\mathbb{R}}

\newcommand{\bbZ}{\mathbb{Z}}

\newcommand{\sfP}{{\sf P}}

\newcommand{\bsx}{\boldsymbol{x}}

\renewcommand{\Im}{\textrm{Im}}

\newcommand{\diff}{\mathrm{d}}
\newcommand{\SLE}{\mathsf{SLE}}

\title{A new computation of pairing probabilities in several multiple-curve models}
\author{
\vspace{-0.1cm}
Alex M. Karrila\thanks{
\AA bo Akademi University, Finland.
E-mail: \texttt{alex.karrila@abo.fi; alex.karrila@gmail.com}
}
}
\date{\vspace{-0.5cm}}

\begin{document}

\maketitle

\begin{abstract}
We give a new, short computation of pairing probabilities for multiple chordal interfaces in the critical Ising model, the harmonic explorer, and for multiple level lines of the Gaussian free field. The core of the argument are the known convexity property and a new uniqueness property of local multiple $\mathsf{SLE}(\kappa)$ measures, valid for all $\kappa > 0$. In particular, the proof is directly is applicable for any underlying random curve model, once it is identified as a local multiple $\mathsf{SLE}(\kappa)$ both conditionally and unconditionally on the pairing topology.
\end{abstract}

\section{Introduction}

\begin{figure}
\begin{center}
\includegraphics[height=3.3cm]{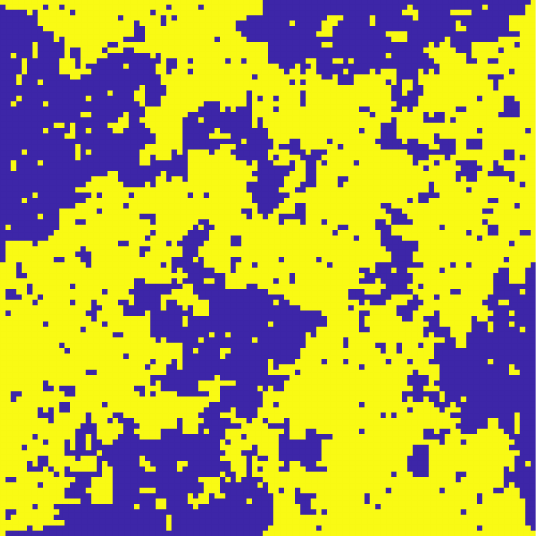}%
\hspace{0.65cm}
\includegraphics[height=3.3cm]{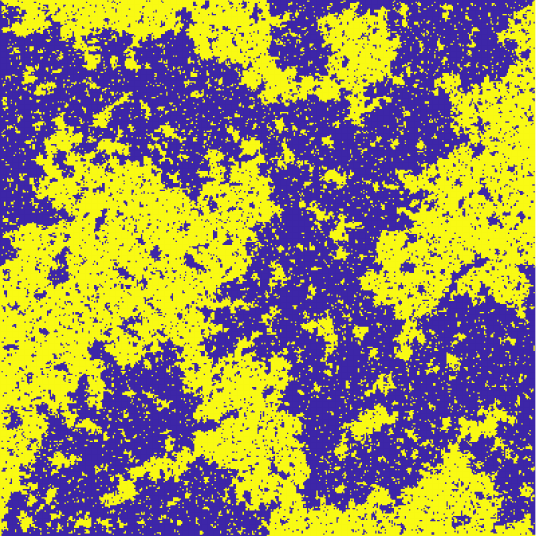}
\hspace{0.65cm}
\includegraphics[height=2.8cm]{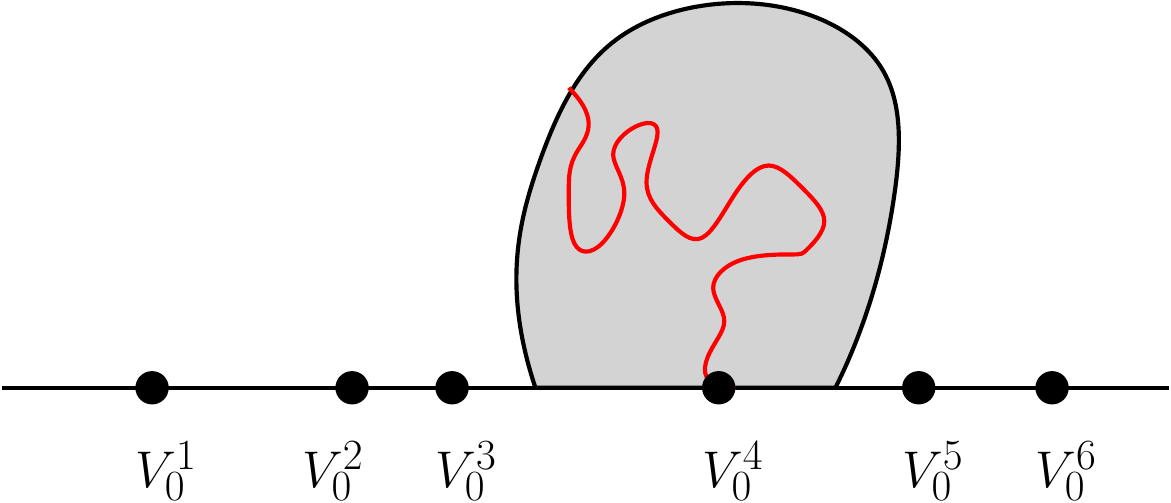}
\end{center}%
\caption{
\label{fig:Ising} 
\label{fig:NSLE schematic}
\textbf{Left, middle:} Simulations of the critical Ising model with alternating boundary conditions on $100 \times 100$ (left) and $400 \times 400$ (middle) square grid graphs. The model with $2N$ boundary condition alternation points naturally gives rise to $N$ chordal interfaces (here $N=4$), pairing up the alternation points in some random manner. Labelling the alternation points in these two examples by the points of compass in the natural manner, the \textit{pairings} of the alternation points by the interfaces are here $\{ \{ \mathrm{sw}, \mathrm{s} \}, \{ \mathrm{se}, \mathrm{e} \}, \{ \mathrm{ne}, \mathrm{n} \}, \{ \mathrm{nw}, \mathrm{w} \} \}$ (left) and $\{ \{ \mathrm{sw}, \mathrm{s} \}, \{ \mathrm{se}, \mathrm{w} \}, \{ \mathrm{e}, \mathrm{ne} \}, \{ \mathrm{n}, \mathrm{nw} \} \}$ (middle). \\
\noindent{\textbf{Right:}}
A schematic illustration of the setup where local multiple SLEs are studied in this note: $2N$ fixed boundary points $V_0^1 < \ldots < V^{2N}_0$, a fixed index $1 \leq j \leq 2N$, and a fixed neighbourhood of $V^j_0$ in $\bbH$ not containing any other marked boundary points on its boundary. The Loewner growth process~\eqref{eq:Loewner ODE}--\eqref{eq:NSLE SDE} starting from $V^j_0$ is then considered up to the stopping time when $K_t$ first hits the boundary (in $\bbH$) of this localization neighbourhood.
}
\end{figure}

Schramm--Loewner evolution (SLE) type curves are conformally invariant random curves \cite{Schramm-LERW_and_UST, RS-basic_properties_of_SLE}
that are known or conjectured to describe (scaling limits of) random interfaces in many critical planar models~(e.g.~\cite{Smirnov-critical_percolation, LSW-LERW_and_UST, SS05,
SS09, 
CDHKS-convergence_of_Ising_interfaces_to_SLE}).
The most common SLE variants are defined in the upper half-plane $\bbH$ via the Loewner equation
\begin{align}
\label{eq:Loewner ODE}
\partial_t g_t (z) &= \frac{2}{g_t(z)-W_t},
\end{align}
whose solution for a given starting point $g_0 (z) = z $ in $ \bbH$ \footnote{We will allow starting points in $\overline{\bbH}$ without explicit mention whenever it is more beneficial.} is only defined up to the (possibly infinite) explosion time when $g_t(z)$ and $W_t$ collide. The sets $K_t$ where the solution is not defined up to time $t$ are those carved out by the initial segment of the SLE curve, while $g_t $ is a conformal map $\bbH \setminus K_t \to \bbH$. For instance, the chordal $\mathsf{SLE}(\kappa)$ from $0$ to $\infty$ in $\bbH$ is obtained by taking $W_t = \sqrt{\kappa} \beta_t$, where $\beta_t$ is a standard Brownian motion and $\kappa > 0$ is the parameter of the model; see~\cite{Lawler-SLE_book} for an introduction.

The most central model in this note is the local multiple $\mathsf{SLE}(\kappa)$~\cite{BBK-multiple_SLEs, Dubedat-commutation, KP-pure_partition_functions_of_multiple_SLEs}
which, e.g., with $\kappa = 3$ describes the scaling-limit interfaces in the Ising model with alternating boundary conditions (see Figure~\ref{fig:Ising}), when mapped conformally to $\bbH$. The definition takes as an input a \textit{partition function}, which is defined as a map $\calZ: \frX \to \bbR_{>0}$, where $\frX = \{  \bsx = (x_1, \ldots, x_{2N} ) \in \bbR^{2N} \; : \; x_1 < \ldots < x_{2N} \} $, satisfying to the conformal covariance condition (denoting $\frac{6 -\kappa}{2 \kappa} = h$)
\begin{align}
\calZ(\bsx ) = \big( \prod_{i=1}^{2N} \phi'(x_i)^h \big) \calZ (\phi (x_1), \ldots, \phi (x_{2N})) 
\end{align}
for all conformal (Möbius) maps $\phi$ from $\bbH$ to $\bbH$ with $\phi(x_1) < \ldots < \phi(x_{2N})$,
and the partial differential equations (PDEs)
\begin{align}
\label{eq:Z-PDEs}
\bigg[ \frac{\kappa}{2 } \partial_{\jmath}^2 + \sum_{\substack{i=1 \\ i \neq \jmath }}^{2N} \Big( \frac{2}{x_i - x_\jmath} \partial_i- \frac{2 h}{(x_i - x_\jmath)^2}\Big) \bigg] \mathcal{Z} ( \bsx ) = 0, \qquad \text{for all } 1 \leq \jmath \leq 2N \text{ and } \bsx \in \frX.
\end{align}
Then, in $\bbH$, the \textit{local multiple SLE} with the $2N$ marked boundary points\footnote{
E.g., for the Ising model, these are the limiting curve endpoints.
Throughout this note, $(V_0^1, \ldots, V^{2N}_0 ) \in \frX$ are regarded as \textit{fixed} by a given geometric setup, while $\bsx \in \frX$ denote ``dummy variables''. 
} 
$V_0^1 < \ldots < V^{2N}_0$, growing from $V^j_0$ is described by the Loewner equation~\eqref{eq:Loewner ODE} where $W_t$ is determined by $W_0 = V^j_0$ and $2N$ coupled SDEs: $V^i_t = g_t (V^i_0)$ for $i \neq j$ and
\begin{align}
\label{eq:NSLE SDE}
\diff W_t = \sqrt{\kappa} \diff \beta_t + \kappa \tfrac{ \partial_j \calZ  }{\calZ  } (V_t^1, \ldots, V_t^{j-1}, W_t, V_t^{j+1}, \ldots, V^{2N}_t) \diff t.
\end{align}


To avoid treating the behaviour of the PDE solutions $\calZ$ at $\partial \frX$, the process~\eqref{eq:NSLE SDE} is usually only studied up to the stopping time when the growing sets $K_t$ hit the boundary of a given localization neighbourhood of $V^j_0$; see Figure~\ref{fig:NSLE schematic} (right). Fixing the geometric setup above,
the laws of two multiple SLE (stopped) driving functions $W_t$ coincide if and only if the partition functions are constant multiples of each other (Lemma~\ref{lem:dr fcns <-> measures}). Note also the set of partition functions is a convex cone; a corresponding convex-space property readily follows for the local multiple SLE measures (Equations~\eqref{eq:tilde-Z}--\eqref{eq:tilde-P}).
The emergence of a convex space, rather than a unique measure, is intuitively explained as non-crossing interfaces in an underlying lattice model can pair up the $2N$ marked boundary points into a Catalan number $C_N = \frac{1}{N+1}\binom{2N}{N}$ of different planar pair partitions, or \textit{pairings}, for short (Figure~\ref{fig:Ising}). Each pairing-conditional measure (as well as their convex combinations) should then converge to a local multiple SLE.

The main contribution of this note is a new, short computation solving the probabilities of the different pairings in several underlying random curve models. The core lemma is the following (we combine the known convexity and the apparently new uniqueness in a single statement).
\begin{lemma}[Convexity and uniqueness properties of multiple SLEs]
\label{lem:ptt fcn conv space ppty} 
Fix launching points $V_0^1 < \ldots < V^{2N}_0$, an index $1 \leq j \leq 2N$, and a localization neighbourhood of $V^j_0$.
In this geometry, let $\sfP$ be a convex combination of finitely many laws $(\sfP_\alpha)_\alpha$ of some multiple $\SLE(\kappa)$ driving functions with respective partition functions $( \calZ_\alpha)_\alpha$:
\begin{align*}
\sfP = \sum_\alpha p_\alpha \sfP_\alpha, \qquad\text{where } p_\alpha \geq 0 \text{ and } \sum_\alpha p_\alpha = 1.
\end{align*}
Then, also $\sfP$ is a local multiple $\SLE(\kappa)$ driving function, and the corresponding partition function $\calZ$, unique up to a multiplicative constant $C > 0$, is given by\footnote{The multiplicative constant clearly satisfies $C = \calZ (V_0^1, \ldots, V^{2N}_0)$. In particular, normalizing the $\calZ_\alpha$:s to attain the value $1$ at $(V_0^1, \ldots, V^{2N}_0) $ and setting analogously $C=1$, we obtain $\calZ (\bsx) = \sum_\alpha p_\alpha \calZ_\alpha (\bsx)$.}
\begin{align}
\label{eq:conv comb ptt fcns}
\calZ (\bsx) = C \sum_\alpha \frac{p_\alpha}{\calZ_\alpha (V_0^1, \ldots, V^{2N}_0) } \calZ_\alpha ( \bsx ) \qquad \text{for all } \boldsymbol{x} \in \frX.
\end{align}
\end{lemma}

Consequently, if some random chordal (scaling limit) curves $\gamma_1, \ldots, \gamma_N$ in $\bbH$ have known descriptions as local multiple SLEs both as such ($\sfP$ and $\calZ$) and conditionally on any of the $C_N$ pairings (which correspond to the indices $\alpha$; $\sfP_\alpha$ and $\calZ_\alpha$), then the pairing probabilities $p_\alpha$ can be solved from the linear coefficients of the $\calZ_\alpha$:s in $\calZ$ (provided that the $\calZ_\alpha$:s are linearly independent). This argument is not specific for any $\kappa$ or any underlying lattice model, but the model's connection to SLEs naturally needs to be established; we explicate the necessary inputs in Section~\ref{subsec:Ising pairing probas}. We then apply this strategy to compute the pairing probabilities in the Ising model, the multiple harmonic explorer, and the level lines of the Gaussian free field. (In all three examples, the identification of the $\alpha$-conditional law $\sfP_\alpha$ is deduced from the theory of the so-called global multiple SLEs~\cite{PW, BPW}, an argument specific for $\kappa \leq 4$.) Whenever solvable by this method, the pairing probabilities $p_\alpha$ become ratios of multiple SLE partition functions. This has interesting interpretations in terms of Conformal field theory (see, e.g.,~\cite{Peltola-towards_a_CFT_for_SLEs}) which we however suppress in this short note. 

The pairing probabilities in the Ising and free field models were earlier solved in~\cite{PW18} and~\cite{PW}, respectively.
Those proofs rely on conditional probability martingales which, unlike the convex combination argument here, have direct and well-known analogues for single SLEs. Analyzing the behaviour of these martingales up to the termination time of the curves (in order to relate their termination value to the correct pairing event) however required in~\cite{PW, PW18} a fine analysis of both the SLE process and the martingale, which were based on technical arguments specific for the corresponding parameter values $\kappa=3$ and $\kappa=4$, respectively. We thus hope that the alternative, short and general proof in this note clarifies the picture. The present proof, especially Lemmas~\ref{lem:dr fcns <-> measures} and~\ref{lem:ex of density}, may also be interesting in the context of other SLE variants with partition functions (see, e.g.,~\cite{BBK-multiple_SLEs, mie2}).\footnote{
Between the appearance of the first preprint and the final version of the present note, another convex independence result for SLE type curves was sketched based on these lemmas in~\cite[Appendix~D.1]{KLPR}.
}
 Other solutions of pairing probabilities of multiple SLE type curves can be found at least in~\cite{Smirnov-critical_percolation, Dubedat-Euler_integrals, KW-boundary_partitions_in_trees_and_dimers, KKP, FPW22}. We conclude by noticing that finding the dimension of the convex cone of multiple SLE partition functions has attracted some attention~(e.g.,~\cite{FK-solution_space_for_a_system_of_null_state_PDEs_1, FK-solution_space_for_a_system_of_null_state_PDEs_2, FK-solution_space_for_a_system_of_null_state_PDEs_3, FK-solution_space_for_a_system_of_null_state_PDEs_4}); Lemma~\ref{lem:ptt fcn conv space ppty} directly connects it to the dimension of the convex space of multiple SLE measures, the natural probabilistic analogue.


\paragraph{Acknowledgements} The author wishes to thank Eveliina Peltola, Lukas Schoug, Lauri Viitasaari, and Hao Wu for discussions and the anonymous referees for valuable insights. Special thanks to Paavo Salminen, who suggested looking for tools for the proof of Lemma~\ref{lem:ex of density} in~\cite{Nualart}. The Academy of Finland (grant \#339515) is gratefully acknowledged for financial support.

\section{Solving the pairings probabilities}

As explained in the introduction, our proof does not rely on properties that are specific for the random model considered, but rather on verifying certain external inputs and then applying Lemma~\ref{lem:ptt fcn conv space ppty}. To facilitate the discussion, we explain below the proof in the context and notation of the Ising model.

\subsection{The Ising model}
\label{subsec:Ising pairing probas}

\paragraph{Notation and domain approximations}
Let $\Omega^\delta$ be simply-connected discrete domains on $\delta \bbZ^2$ with marked boundary points $ p_1^\delta, \ldots, p_{2N}^\delta$. Let $\bbP^{\Omega^\delta}$ denote the critical Ising model on $(\Omega^\delta; p_1^\delta, \ldots, p_{2N}^\delta)$ with alternating $+$ and $-$ boundary conditions; 
see Figure~\ref{fig:Ising} for an illustration and Appendix~\ref{app:Ising} for the precise definition.
We say that the discrete domains with marked boundary points $(\Omega^\delta; p_1^\delta, \ldots, p_{2N}^\delta)$ on $\delta \bbZ^2$ converge in the \textit{Carath\'{e}odory sense} if the conformal maps $\varphi^\delta: \bbH \to \Omega^\delta $ (normalized in some fixed manner, e.g., at the point $i \in \bbH$) converge to some $\varphi$ uniformly over compact subsets of $\bbH$ and also the inverse images of the boundary points converge: $(\varphi^\delta)^{-1}(p_i^\delta) \to V_0^{i}$ for $1 \leq i \leq 2N$, for some $V^1_0, \ldots, V^{2N}_0 \in \bbR$;
we also assume that $\varphi(\bbH)$ is bounded (cf.~\cite{Chelkak-Smirnov:Discrete_complex_analysis_on_isoradial_graphs}) and that the limit points $V_1^{0}, \ldots, V_0^{2N}$ are distinct and labelled so that $V_0^{1} < \ldots < V_0^{2N}$.  
Let $\gamma^\delta_1, \ldots, \gamma^\delta_N$ denote the chordal curves in $\bbH$ obtained by mapping the Ising interfaces on $(\Omega^\delta; p_1^\delta, \ldots, p_{2N}^\delta)$ by $(\varphi^\delta)^{-1}$.  The Ising model conditional on a pairing, indexed here by $\alpha$, is denoted below by $\bbP^{\Omega^\delta} [ \; \cdot \; | \; $pairing $ \alpha] =: \bbP^{\Omega^\delta}_\alpha [ \; \cdot \; ]$.

\paragraph{External inputs}

The proof assumes the following:
\begin{itemize}
\item[i)] For any sequence $\delta \to 0$, there exists a subsequence along which the curves $\gamma^\delta_1, \ldots, \gamma^\delta_N$ under $\bbP^{\Omega^\delta}$ (resp. under $\bbP^{\Omega^\delta}_\alpha$, for all pairings $\alpha$ appearing with $\bbP^{\Omega^\delta}$-positive probability for arbitrarily small values of $\delta$) converge weakly
in the metric of unparametrized curves~\cite{KS}.
\item[ii)] For some starting point $V^j_0$ and in some localization neighbourhood, the initial segment of incident limit curve of the $\bbP^{\Omega^\delta}$:s (resp. of all the above $\bbP^{\Omega^\delta}_\alpha$:s) is a local multiple $\SLE(\kappa)$ with partition function $\calZ$ (resp. $\calZ_\alpha$).
\item[iii)] There is an explicit expression of $\calZ$ as a linear combination of $\calZ_\alpha$:s.
\end{itemize}

\paragraph{Checking the inputs}
For the Ising model, input (i) follows from the weak convergence criterion~\cite[Theorem~1.5]{KS} which is verified in~\cite[Corollary~1.7]{CDCH} (see also~\cite[Theorem~4.1]{mie}). As for input (ii), the local limit identification for $\bbP^{\Omega^\delta}$ is proven in~\cite[Theorem~1.1]{Izyurov-critical_Ising_interfaces_in_multiply_connected_domains} (by using input (i) to deduce the existence of subsequential limits and a martingale observable to identify them).  The SLE parameter is $\kappa = 3$ and the partition functions $\calZ: \frX \to \bbR_{>0}$ is given by the Pfaffian formula
\begin{align}
\label{eq:Ising ptt fcn}
\calZ ( \bsx ) = \mathrm{Pf} \left( \tfrac{1}{x_j - x_i} \right)_{i, j = 1}^{2N}
\end{align}
(where the diagonal elements of the matrix should be interpreted as zeroes).
For the conditional measures $\bbP^{\Omega^\delta}_\alpha$ for any $\alpha$, any subsequential limit curves (as in input~(i)) are identified as so-called \textit{global multiple SLEs} with pairing $\alpha$ (this proof essentially boils down to the uniqueness of the latter) in~\cite{BPW}. The initial segments of the latter were described in~\cite[Theorem~1.3]{PW} as $\kappa = 3$ local multiple SLEs with the partition function $\calZ_\alpha$ given in~\cite[Equation~(3.7)]{PW}, concluding input~(ii) for the measures $\bbP^{\Omega^\delta}_\alpha$. As regards~(iii), we have $ \sum_\alpha \mathcal{Z}_\alpha (\bsx) = \mathcal{Z} (\bsx) $ by~\cite[Lemma~4.13]{PW} where the sum is over all $C_N$ pairings.

We note in passing that the earlier proof of Theorem~\ref{thm:Ising pairing probas} in~\cite{PW18} requires inputs~(i) and~(iii), as well as input~(ii) for the unconditional measures $\bbP^{\Omega^\delta}$ (but not for the $\bbP^{\Omega^\delta}_\alpha$:s).\footnote{Unlike the present proof, the one in~\cite{PW18} additionally relies on explicit expressions for certain GFF level line connection probabilities (sums of those in Section~\ref{subsec:GFF lvl lines});~see Appendix A in \cite{PW18}.}

\begin{theorem}[Pairing probabilities in the critical Ising model]
\label{thm:Ising pairing probas}
Let $(\Omega^\delta; p_1^\delta, \ldots, \newline p_{2N}^\delta)$ on $\delta \bbZ^2$ converge in the Carath\'{e}odory sense as $\delta \downarrow 0$; then, we have
\begin{align}
\label{eq:main thm}
\bbP^{\Omega^\delta} [ \; \mathrm{pairing} \; \alpha \; ] \longrightarrow  \mathcal{Z}_\alpha ( V^1_0, \ldots, V^{2N}_0 ) \Big/ \mathcal{Z} ( V^1_0, \ldots, V^{2N}_0 )  \qquad \text{as } \delta \downarrow 0,
\end{align}
where $\mathcal{Z} $ and $\mathcal{Z}_\alpha $ are the $\kappa=3$ multiple SLE partition functions defined in Equation~\eqref{eq:Ising ptt fcn} and in~\cite[Equation~(3.7)]{PW}, respectively.
\end{theorem}

\begin{proof}
Suppose, by input~(i), that a subsequence of $\delta$'s has already been extracted so that $( \gamma^\delta_1, \ldots, \gamma^\delta_N )$ converge weakly
under $\bbP^{\Omega^\delta}$ as well as under $\bbP^{\Omega^\delta}_\alpha$ for all pairings $\alpha$ appearing with $\bbP^{\Omega^\delta}$-positive probability for arbitrarily small $\delta$:s. Denote the limiting laws by $\bbP$ and $\bbP_\alpha$, respectively. It clearly suffices to prove the claimed limit~\eqref{eq:main thm} along such a subsequence.

The connection probabilities at least converge to some numbers along this subsequence:
$$\bbP^{\Omega^\delta} [ \; \mathrm{pairing} \; \alpha \; ] \to  \bbP [  \; \mathrm{pairing} \; \alpha \; ] =: p_\alpha \in [0,1]. $$
Hence we also have the convex combination formula
\begin{align}
\label{eq:conv comb meas}
\bbP = \sum_\alpha p_\alpha \bbP_\alpha
\end{align}
for the limiting measures, where the sum runs over the above mentioned $\alpha$:s. By input (ii), under $\bbP$ (resp. $\bbP_\alpha$), the driving function of the initial segment of the curve starting from a fixed boundary point in a fixed localization neighbourhood has the local multiple SLE law $\sfP$ (resp. $\sfP_\alpha$) with a known partition function $\calZ$ (resp. $\calZ_\alpha$). From~\eqref{eq:conv comb meas} and Lemma~\ref{lem:ptt fcn conv space ppty}, we infer that
\begin{align*}
\mathcal{Z} ( \bsx ) 
= \mathcal{Z} (V^1_0, \ldots, V^{2N}_0 ) \sum_\alpha \tfrac{ p_\alpha }{ \mathcal{Z}_\alpha (V^1_0, \ldots, V^{2N}_0 ) } \mathcal{Z}_{\alpha} ( \bsx ).
\end{align*}
Finally, by input (iii), for all models in this note with the formula $\mathcal{Z}( \bsx )  = \sum_\alpha \mathcal{Z}_\alpha ( \bsx )$ where the sum is over all pairings, and with all the functions $\mathcal{Z}_\alpha$ being linearly independent\footnote{In all our proofs, and conjecturally for all multiple SLE convergences, the $\calZ_\alpha$:s corresponding to given pairings are so-called the \textit{SLE pure partition functions}; their linear independence follows
from their defining asymptotic property~\cite[Eq.~(ASY)]{PW}, by induction over $N$.}, we deduce
\begin{align*}
p_\alpha =  \mathcal{Z}_\alpha (V^1_0, \ldots, V^{2N}_0 ) \big/ \mathcal{Z} ( V^1_0, \ldots, V^{2N}_0 ) , \qquad \text{for all pairings }\alpha.
\end{align*}
This concludes the proof.
\end{proof}

\subsection{The multiple harmonic explorer}

We now briefly overview an analogous result for the multiple harmonic explorer introduced in~\cite{mie}. Let $H$ denote the honeycomb lattice, let $\Omega$ consist of the faces on or inside a simple loop path on the faces of $H$ (i.e., on the dual graph $H^*$), and colour the faces on this loop into $N$ white and $N$ black segments, changing colours at $2N$ boundary points $p_1, \ldots, p_{2N}$. The harmonic explorer explores an interface between black and white faces, starting from the edge emanating from $p_1$ into $\Omega$: if the face right in front of this edge is already coloured, the next edge of the interface is evident; otherwise, launch a dual random walk from this face, colour the face according to the first coloured face hit by the walk, and then deduce the next edge. Iteratively, one then adds new edges and new coloured faces, e.g., in a circular order to the boundary points, until $N$ entire chordal interfaces have been revealed. We denote by $\bbP^{\Omega}$ probability measure of the harmonic explorers on $(\Omega; p_1, \ldots, p_{2N})$.

\begin{theorem}[Pairing probabilities of multiple harmonic explorers]
Let $(\Omega^\delta; p_1^\delta, \newline \ldots, p_{2N}^\delta)$ be discrete domains on the scaled lattice $\delta H$, as described above, and converging in the Carath\'{e}odory sense as $\delta \downarrow 0$; then, we have
\begin{align*}
\bbP^{\Omega^\delta} [ \; \mathrm{pairing} \; \alpha \; ] \longrightarrow  \mathcal{Z}_\alpha ( V^1_0, \ldots, V^{2N}_0 ) \Big/ \mathcal{Z} ( V^1_0, \ldots, V^{2N}_0 ) \qquad \text{as } \delta \downarrow 0,
\end{align*}
where $\mathcal{Z}_\alpha $ and $\mathcal{Z} $ are the $\kappa=4$ multiple SLE partition functions defined in Equation~\eqref{eq:GFF ptt fcn} and in~\cite[Equation~(3.7)]{PW}.
\end{theorem}

As explained above, we just need to verify the inputs of the previous proof. Input (i), as well as input (ii) for the unconditional measures $\bbP^{\Omega^\delta}$ are proven in~\cite[Theorem~6.10]{mie}; the SLE parameter is $\kappa = 4$ and partition function is
\begin{align}
\label{eq:GFF ptt fcn}
\mathcal{Z} ( \bsx ) 
 = \prod_{1 \leq i < j \leq 2N} (x_j - x_i)^{(-1)^{j-i}/2}.
\end{align}
Input (ii) for the conditional measures $\bbP^{\Omega^\delta}_\alpha$ is~\cite[Theorem~5.10]{mie} (whose proof relies on the global multiple SLEs similarly to the Ising model case), and the partition functions $\calZ_\alpha$ are as defined in~\cite[Equation~(3.7)]{PW} with $\kappa = 4$.
Finally, input~(iii), incidentally with the exact same relation 
\begin{align}
\label{eq:GFF cvx space}
\mathcal{Z} ( \bsx )  = \sum_\alpha \mathcal{Z}_\alpha ( \bsx ),
\end{align}
as for the Ising model, is proven in~\cite[Lemma~4.14]{PW}.

\subsection{The level lines of the Gaussian free field}
\label{subsec:GFF lvl lines}

We yet explicate the analogous proof for multiple level lines of the Gaussian free field (GFF).
Let $\bbP$ be the GFF measure (see Appendix~\ref{app:GFF}) in $\bbH$ with the following alternating boundary conditions: given $V^1_0 < \ldots < V_0^{2N}$, the boundary condition at $x \in \bbR$ is set to be $\lambda$ if the number of marked boundary points strictly left of $x$ is even, and $- \lambda$ if odd; here we denote $\lambda = \sqrt{\pi / 8}$, following the GFF normalization convention of~\cite{Werner-GFF}. 
While the GFF cannot be represented as a continuous function, and it therefore has no level lines in the usual sense, level lines do exist in the sense of a suitable coupling (see Proposition~\ref{prop:GFF lvl lines} in the appendix). In particular, these level lines are disjoint chordal curves between the boundary points $V^1_0, \ldots, V_0^{2N}$, hence forming some pairing.

\begin{theorem}[Pairing probabilities of GFF level lines]
\label{thm:GFF lvl lines}
The pairing probabilities of the GFF level lines are given by
\begin{align*}
\bbP [ \; \mathrm{pairing} \; \alpha \; ] =  \mathcal{Z}_\alpha ( V^1_0, \ldots, V^{2N}_0 ) \Big/ \mathcal{Z} ( V^1_0, \ldots, V^{2N}_0 ) ,
\end{align*}
where $\mathcal{Z}_\alpha $ and $\mathcal{Z} $ are the $\kappa=4$ multiple SLE partition functions defined in Equation~\eqref{eq:GFF ptt fcn} and in~\cite[Equation~(3.7)]{PW}.
\end{theorem}

The proof is still essentially identical. Since the underlying model is continuous \textit{per se}, one can either start the proof from Equation~\eqref{eq:conv comb meas}, or add a nonsensical ``mesh size'' parameter $\delta$ and set all the probability measures to be the same independently of $\delta$. Note that here we also have to restrict $\alpha$ to those pairings that are possible (due to the precise form of Corollary~\ref{cor:GFF lvl lines}; \textit{a posteriori} by Theorem~\ref{thm:GFF lvl lines} all pairings are possible). As regards checking the inputs,~(i) is trivial. Input~(ii) for the unconditional measures is Proposition~\ref{prop:GFF lvl lines}(v) in the Appendix. For the $\alpha$-conditional measures, the GFF level lines are the $\alpha$ global multiple $\mathsf{SLE}(4)$ by Corollary~\ref{cor:GFF lvl lines}, and any initial segment of the latter is by~\cite[Theorem~1.3]{PW} the local multiple $\mathsf{SLE}(4)$ with the partition function $\calZ_\alpha$ defined in~\cite[Equation~(3.7)]{PW}. Input~(iii) is Equation~\eqref{eq:GFF cvx space} above.

\begin{remark}
We have for simplicity considered here a boundary condition alternating between $+ \lambda$ and $-\lambda$, as in~\cite{PW}. \emph{Mutatis mutandis}, a similar proof applies for any boundary condition with $N$ jumps $+ 2 \lambda$ and $N$ jumps $- 2 \lambda$. This provides a concise proof of~\cite[Theorem~4.1]{Liu-Wu} (and an example of input (iii) not being of the form $\mathcal{Z} ( \bsx )  = \sum_\alpha \mathcal{Z}_\alpha ( \bsx )$).
\end{remark}

\section{The SLE theory lemmas}

The following uniqueness lemma is actually the main new contribution in Lemma~\ref{lem:ptt fcn conv space ppty}, and crucial in our applications. We highlight that we study local multiple SLEs in a fixed geometry, not as collections of measures indexed by geometries --- in the latter case the analogous lemma is immediate~\cite[Theorem~A.4(a)]{KP-pure_partition_functions_of_multiple_SLEs}.

\begin{lemma}[Uniqueness of the multiple SLE partition function]
\label{lem:dr fcns <-> measures}
Fix the launching points $V_0^1 < \ldots < V^{2N}_0$, an index $1 \leq j \leq 2N$, and a localization neighbourhood of $V^j_0$. Let $ \mathcal{Z}_1 $ and $\mathcal{Z}_2$ be two local multiple $\SLE(\kappa)$ partition functions and $\sfP_1$ and $\sfP_2$ the corresponding laws of the driving function. Then, $\sfP_1$ and $\sfP_2$ are equal if and only if $ \mathcal{Z}_1 $ and $\mathcal{Z}_2$ are constant multiples of each other.
\end{lemma}

The rest of this section constitutes the proofs of Lemmas~\ref{lem:ptt fcn conv space ppty} and~\ref{lem:dr fcns <-> measures}; we first prove the former (assuming the latter).


\begin{proof}[Lemma~\ref{lem:ptt fcn conv space ppty}]
 First, any convex combination of partition functions
\begin{align}
\label{eq:tilde-Z}
\tilde{\mathcal{Z}} ( \bsx ) = \sum_\alpha c_\alpha \mathcal{Z}_\alpha ( \bsx ),
\end{align}
where $c_\alpha \geq 0 $ with $\sum_\alpha c_\alpha = 1$ are constants (independent of $\bsx$), is clearly also a partition function. By~\cite[Theorem~A.4(c)]{KP-pure_partition_functions_of_multiple_SLEs} (whose proof is a short computation with the Girsanov martingales~\eqref{eq:Girs mg 1} below), the law $\tilde{\sfP}$ of the driving function corresponding to the partition function $\tilde{ \mathcal{Z}}$ and the launching points $V^1_0, \ldots, V^{2N}_0$, is a convex combination of those corresponding to $\mathcal{Z}_\alpha$:
\begin{align}
\label{eq:tilde-P}
\tilde{ \sfP } = \sum_\alpha \tfrac{ c_\alpha \mathcal{Z}_\alpha (V^1_0, \ldots, V^{2N}_0 ) }{ \tilde{\mathcal{Z}} (V^1_0, \ldots, V^{2N}_0 ) } \sfP_\alpha.
\end{align}

Coming back to the setup of the present lemma, let us consider the specific convex combination of partition functions with 
\begin{align}
\label{eq:coeffs in tilde-Z}
c_\alpha = C' \tfrac{ p_\alpha }{ \mathcal{Z}_\alpha (V^1_0, \ldots, V^{2N}_0 ) }, 
\end{align}
with a normalizing constant $C'$ matched so that $\sum_\alpha c_\alpha = 1$; the corresponding convex combination measure thus becomes 
\begin{align*}
\tilde{ \sfP }  = \sum_\alpha C' \tfrac{ p_\alpha  }{ \tilde{\mathcal{Z}} ( V^1_0, \ldots, V^{2N}_0 ) } \sfP_\alpha = \sum_\alpha  p_\alpha  \sfP_\alpha ,
\end{align*}
where we observed that $ C'  = \tilde{\mathcal{Z}} (V^1_0, \ldots, V^{2N}_0 ) $ by~\eqref{eq:tilde-Z} and \eqref{eq:coeffs in tilde-Z}. 
The convex combination $\sfP = \sum_\alpha  p_\alpha  \sfP_\alpha $ thus coincides with a multiple SLE law $\tilde{\sfP}$ with the partition function~\eqref{eq:tilde-Z}--\eqref{eq:coeffs in tilde-Z}. Since, by Lemma~\ref{lem:dr fcns <-> measures}, the law of a multiple SLE determines its driving function up to constant multiplication, this identifies the partition function of $\sfP$ as~\eqref{eq:conv comb ptt fcns}
\end{proof}


\begin{proof}[Lemma~\ref{lem:dr fcns <-> measures}]
The implication ``if'' is obvious, so we concentrate on ``only if''. 
Let us
assume for a lighter notation that $j=1$ and that all the processes are stopped at the exit time of the localization neighbourhood, but omit denoting for this stopping. Let $(\mathscr{F}_t)_{t \geq 0}$ be the natural right-continuous filtration of the driving function $W_t$ (hence stopped).

First, denoting $\rho = \frac{ \mathcal{Z}_2 }{\mathcal{Z}_1}$, it is fairly standard that
\begin{align}
\label{eq:Girs mg 1}
M_t :=  \rho (W_t, V^2_t, \ldots, V^{2N}_t )
\end{align}
is the measure-changing martingale from $\sfP_1$ to $\sfP_2$, i.e.,
\begin{align}
\label{eq:RNder}
\tfrac{\mathrm{d} \mathsf{P}_2 }{ \mathrm{d} \mathsf{P}_1 } \vert_{\mathscr{F}_t} = M_t/M_0.
\end{align}
Indeed, the PDEs~\eqref{eq:Z-PDEs} for the two partition functions imply by direct computation (see~\cite[Theorem~1]{Duffin} for a clever way) that $\rho $ satisfies
\begin{align}
\label{eq:rho-PDEs}
\Bigg[ \frac{\kappa}{2 } \partial_{\jmath}^2 + \kappa \frac{\partial_{\jmath} \mathcal{Z}_1 (  \bsx ) }{ \mathcal{Z}_1 ( \bsx )} \partial_{\jmath} + \sum_{\substack{i = 1 \\ i \neq \jmath}}^{2N}  \frac{2}{x_i - x_{\jmath} } \partial_i \Bigg] \rho ( \bsx ) = 0, \qquad \forall 1 \leq \jmath \leq 2N,
\end{align}
which with $\jmath=1$ shows that the It\^{o} drift term of $M_t$ vanishes under $\sfP_1$. Next, $M_t$ can be shown bounded up to the stopping time (hence a genuine martingale) by the translation invariance of $\rho$ and a standard harmonic measure argument (e.g.,~\cite[Proof of Lemma~B.1]{mie2}). Equation~\eqref{eq:RNder} then follows from Girsanov's theorem.
In particular, the two measures are equal if and only if $M_t$ is $\sfP_1$-almost surely a constant process.

Now,
summing the PDEs~\eqref{eq:rho-PDEs}
over all $\jmath$ yields an elliptic PDE that is hence also satisfied by $\rho$. The strong maximum principle of elliptic PDEs then states that $ \rho $ is either an everywhere constant function, or it is nowhere locally constant (cf.~\cite[Theorem~2]{Duffin}).
In the former case the proof is finished, so the rest of the proof consists of assuming the latter and showing that $M_t$ then cannot be almost surely a constant process.

Let $\sfP$ be a third measure on driving functions, under which $W_\cdot$ has the law of a Brownian motion launched from $V^1_0$ and running at speed $\kappa$, i.e., $V^1_0 + \sqrt{\kappa} \beta_t$, but with the same stopping as throughout. Yet another Girsanov transform is
\begin{align*}
\frac{\mathrm{d} \mathsf{P}_1 }{ \mathrm{d} \mathsf{P} } \vert_{\mathscr{F}_t} =  \calZ_1 (W_t, V^2_t, \ldots, V^{2N}_t) \prod_{i=2}^{2n} \left( g_t'(V^i_0) \right)^h ,
\end{align*}
where $g_t'$ denotes the derivative of the Loewner mapping-out function. It thus suffices to show that $\rho (W_t, V^2_t, \ldots, V^{2N}_t )$ is not a constant process under $\sfP$. Note that typical SDE results are derived for Lipschitz SDEs, while $\diff V^{i}_t = \frac{2 \diff t}{V^{i}_t - W_t}$ is not Lipschitz if small $| V^{i}_t - W_t |$ are allowed. However, up to the stopping time, the above-mentioned harmonic measure argument lower-bounds $| V^{i}_t - W_t |$. Thus, rather than $\sfP$, we will actually perform the SDE analysis on the \textit{non-stopped} process
\begin{align}
\label{eq:SDE smoothed}
\begin{cases}
\diff \tilde{V}^1_t = \sqrt{\kappa} \diff \beta_t \\
\diff \tilde{V}^j_t = \theta ( \tilde{V}^{j}_t - \tilde{V}^1_t ) \frac{2 }{\tilde{V}^j_t - \tilde{V}^1_t} \diff t, \qquad 2 \leq j \leq 2N,
\end{cases}
\end{align}
with $(\tilde{V}^1_0, \ldots, \tilde{V}^{2N}_0) = ( V^1_0, \ldots, V^{2N}_0)$,
where $\theta$ is a smooth cutoff function being one in $\bbR \setminus (- \epsilon, \epsilon)$ and zero in $[- \epsilon/2, \epsilon /2 ]$, with a small enough $\epsilon > 0$ so that up to the stopping time, the process $(\tilde{V}^1_t, \ldots, \tilde{V}^{2N}_t) $ is the same as $(W_t, V^2_t, \ldots, V^{2N}_t )$ under $\sfP$ (here and below, we sample the two from the same Brownian motion).

\begin{lemma}
\label{lem:ex of density}
For all fixed $t > 0$, the law of $(\tilde{V}^1_t, \ldots, \tilde{V}^{2N}_t)$, as defined above, is absolutely continuous with respect to the Lebesgue measure on $\bbR^{2N}$.
\end{lemma}

Assuming Lemma~\ref{lem:ex of density} for a moment, the proof is readily finished: fix $T$ small enough so that the event $E_T$ that the exit time has not occurred by $T$ has a positive probability under $\sfP$. On the event $E_T$, we have $(W_T, V^2_T, \ldots, V^{2N}_T ) = (\tilde{V}^1_T, \ldots, \tilde{V}^{2N}_T)$. On the other hand, the set $\{ \bsx \in \frX \; : \; \rho ( \bsx ) = \rho (V^1_0, \ldots, V^{2N}_0) \}$ was assumed to have Lebesgue measure zero, so by Lemma~\ref{lem:ex of density}, we have $\rho (\tilde{V}^1_T, \ldots, \tilde{V}^{2N}_T) \neq \rho ( V^1_0, \ldots, V^{2N}_0)$ with probability one. 
 Hence, $\sfP [ \rho (W_T, V^2_T, \ldots, V^{2N}_T ) \neq \rho (V^1_0, \ldots, V^{2N}_0) ] \geq \sfP [E_T] > 0$.
\end{proof}

\begin{proof}[Lemma~\ref{lem:ex of density}]
The proof is based on the Hörmander criterion in~\cite[Theorem~2.3.1]{Nualart}.\footnote{The cutoff function $\theta$ was introduced in the previous proof exactly to produce a smooth SDE with Lipschitz coefficients, as required there.} To coincide with the notations there, denote $\beta_t = \beta^1_t$ and re-write~\eqref{eq:SDE smoothed} as
\begin{align*}
\diff \tilde{\mathbf{V}}_t = \mathbf{A}_1 (\tilde{\mathbf{V}}_t) \diff \beta^1_t + \mathbf{B} (\tilde{\mathbf{V}}_t)\diff t,
\end{align*}
where $\mathbf{A}_1: \bbR^{2N} \to \bbR^{2N}$ is here given by the constant function $\mathbf{A}_1 = (\sqrt{\kappa}, 0, \ldots 0)$ 
and $\mathbf{B}: \bbR^{2N} \to \bbR^{2N}$ by $\mathbf{B} ( \bsx ) = (0, \theta ( x_2 - x_1 ) \frac{2 }{ x_2 - x_1 }, \ldots, \theta ( x_{2N} - x_1 ) \frac{2 }{ x_{2N} - x_1})$. Lemma~\ref{lem:ex of density} can then be deduced by checking that the vector fields (operators)
\begin{align*}
A_1; \quad [A_i, A_j] \text{ with } 0 \leq i, j \leq 1; \quad [A_i, [A_j, A_k]] \text{ with } 0 \leq i, j, k \leq 1; \quad \ldots
\end{align*}
at the launching point $(V^1_0, \ldots, V^{2N}_0)$ span a $2N$-dimensional vector space; here the operator commutator is the usual one and the operators are in our case
\begin{align*}
A_1 &= \sum_{i=1}^{2N} \mathbf{A}^i_1 (\bsx) \partial_i =
\sqrt{\kappa} \partial_1 \qquad \text{and} \\
A_0 &= \sum_{i=1}^{2N} \Big( \mathbf{B}^i (\bsx )  + \tfrac{1}{2} \sum_{j=1}^{2N} \mathbf{A}^j_1 ( \bsx ) \big( \partial_j \mathbf{A}^i_1 ( \bsx ) \big) \Big) \partial_i  = \sum_{i=2}^{2N} \theta ( x_i - x_1 ) \tfrac{2 }{ x_i - x_1} \partial_i.
\end{align*}

We now verify this criterion.\footnote{In technical terms, this verification is very similar to~\cite[Proof of Proposition~2.6]{PW}. The logical connection, however, is less direct: \cite{PW} studies the \textit{spatial PDEs}~\eqref{eq:Z-PDEs} via the original Hörmander criterion, while the variant in~\cite{Nualart} originates in the \textit{spatio-temporal (Fokker--Planck--Kolmogorov) PDEs}. The superficial similarity stems from the duality of the spatial-variable parts in these differential operators.} Define the vector fields $G_1 = [A_1, A_0]$ and $G_k = [A_1, G_{k-1}]$ for $k \geq 2$. By direct computation, we have
\begin{align*}
G_k ( V^1_0, \ldots, V^{2N}_0) = C_k  \sum_{i=2}^{2N} \tfrac{2}{ (V^i_0 - V^1_0 )^{k + 1}} \partial_i,
\end{align*}
where $C_k > 0$ are constant factors that are irrelevant in what follows.
In matrix algebra:
\begin{align*}
\left(
\begin{matrix}
G_1 ( V^1_0, \ldots, V^{2N}_0) \\
\vdots \\
G_{2N-1} ( V^1_0, \ldots, V^{2N}_0)
\end{matrix}
\right) 
= \Lambda M \left(
\begin{matrix}
\partial_2 \\
\vdots \\
\partial_{2N}
\end{matrix}
\right),
\end{align*}
where $\Lambda$ is a diagonal matrix, $\Lambda_{k, k} = C_k$, and $M$ is a Vandermonde (type) matrix $M_{k,i} = \frac{2 }{ (V^i_0 - V^1_0 )^{k + 1}}$. Hence, one can invert these two matrices and express the operators $\partial_2,
\ldots,
\partial_{2N}$ as linear combinations of $G_1, \ldots, G_{2N-1}$. Since in addition $A_1 \propto \partial_{1}$, the Hörmander criterion is satisfied.
\end{proof}

\begin{appendix}

\section{Precise definitions of the models}

Since our proofs did not refer to any properties specific for the random models underlying the curves, we have postponed the precise definitions to this appendix. The role of these definitions is to give a definite setup under which the inputs necessary for our results hold.

\subsection{The Ising model}
\label{app:Ising}

 Let $\Omega$ be a planar domain, bounded by a simple closed curve on $\bbZ^2$. Let $p_1, \ldots, p_{2N} \in \bbZ^2$ be distinct vertices on $\partial \Omega$. For all edges $e$ of $\bbZ^2$ on $\partial \Omega$, we give a spin boundary condition $\tau_e = +1$ if $e$ lies on one of the arcs $p_1 p_2, p_3 p_4, \ldots, p_{2N-1}p_{2N}$, and $\tau_e = -1$ otherwise. A spin configuration on (the faces of) $\Omega$ is then a map  $\sigma: (\bbZ^2)^* \cap \Omega \to \{ \pm 1\}$ from the faces to signs $\pm 1$. The Hamiltonian of a configuration is 
\begin{align*}
 H(\sigma) = - \sum_{u \sim v} \sigma_u \sigma_v - \sum_{u \sim e} \sigma_u \tau_e,
 \end{align*}
where the sums run over all face--face or face--boundary edge adjacency pairs, respectively. The Ising model is then finally defined via
\begin{align*}
\bbP^\Omega [\sigma] \propto e^{-\beta H(\sigma)}, \quad \text{ where } \quad
\beta = \beta_{crit} = \tfrac{1}{2} \log (1 + \sqrt{2})
\end{align*}
is the critical (inverse) temperature fixed throughout this note.

To define the interface curves, explore from each $p_i$ with $i$ odd, the interface on $ \bbZ^2$ between $+1$ and $-1$ spins, taking the left-most alternative whenever there is an ambiguity. By this left-turning rule, on the left of the interface there is a $(\bbZ^2)^*$-path of spins $-1$, forming the outer boundary of the $-1$ cluster adjacent to the negative boundary conditions. Consequently, such interfaces are mutually non-crossing, edge-simple, chordal interfaces on $\overline{\Omega} \cap \bbZ^2$, pairing the odd-index and even-index boundary points. 

\subsection{The Gaussian free field and its level lines}
\label{app:GFF}

Let $D \neq \bbC$ be a simply-connected planar domain and $G_D: D \times D \to \bbR$ the Green's function of the negative Laplacian operator $-\Delta$ on $D$.\footnote{Explicitly,
$
G_\bbH(z, w) = \frac{1}{2 \pi} ( \log \left| {z - w^*} \right| - \log \left| {z-w} \right|),
$
and if $\phi$ is a conformal map $\bbH \to D$, then
$
G_D (x, y) = G_\bbH (\phi^{-1} (x), \phi^{-1} (y)).
$
}
Let $\calM_D$ be the set of finite signed Borel measures $\mu$ supported on $D$ with
\begin{align*}
\int_{D \times D} G_D (x, y) \diff \mu(x) \diff \mu (y) < \infty
\end{align*}
(this is satisfied, e.g., if $\mu$ is absolutely continuous with respect to the Lebesgue measure).
The Gaussian free field (GFF) $\Gamma_D$ on $D$ with Dirichlet boundary conditions is a centered Gaussian process indexed by 
$\calM_D$ and determined by the covariance structure
\begin{align*}
\mathrm{Cov} (\Gamma_D(\mu), \Gamma_D (\nu)) = \int_{D \times D} G_D (x, y) \diff \mu(x) \diff \nu (y).
\end{align*}
If $H$ is a harmonic function on $D$ with regular enough boundary behaviour,\footnote{
In this paper, we may require that $H$ is bounded and the limits of $H \circ \phi: \bbH \to \bbR$ exist at all but finitely many real points. Generally, GFF theory often requires allowing much more general functions $H$.
}
then the GFF $\hat{\Gamma}_D$ with boundary condition $H_{| \partial D}$ is given by
\begin{align*}
\hat{\Gamma}_D (\mu) = \Gamma_D (\mu) + \int_D H(x) \diff \mu (x).
\end{align*}
We refer the reader to the textbook~\cite{Werner-GFF} for a more thorough introduction to the GFF; particular key features in the brief discussion below are its conformal property and the concept of local sets. Roughly speaking, the latter refers to a random sets on which the value of the GFF can be explored and satisfies a strong spatial Markov property, just like the strong Markov property of the Brownian motion only refers to stopping times and not any random times.

It is by now standard that the GFF has ``level lines'' or ``cliff lines'', in the sense of suitable local sets that turn out to be (initial segments of) $\SLE(4)$ type curves.\footnote{
Unlike in the rest of this note, in this appendix \textit{initial segments} of SLEs refer to growth processes stopped at a any stopping time (not only the exit time of a fixed neighbourhood) that almost surely occurs before any marked point is swallowed by the hulls $K_t$.
Also, we have so far studied SLEs only through their driving functions while local sets are random compact sets in the Hausdorff set distance. The above hulls, however, depend continuously on the driving function, for the almost sure restriction of $\SLE(4)$ onto Loewner evolutions generated by continuous simple curves~\cite{KS}, so this distinction is immaterial.
}
Note however that it is crucial for us to work with \textit{collections} of \textit{full} curves, not just initial segments of a single curve as was the case in the first level line statements~(e.g.~\cite{IG1}). Since we are not aware of a direct reference, we give the precise statement and collect the appropriate references as a proof.



\begin{figure}
\begin{center}
\includegraphics[width=8cm]{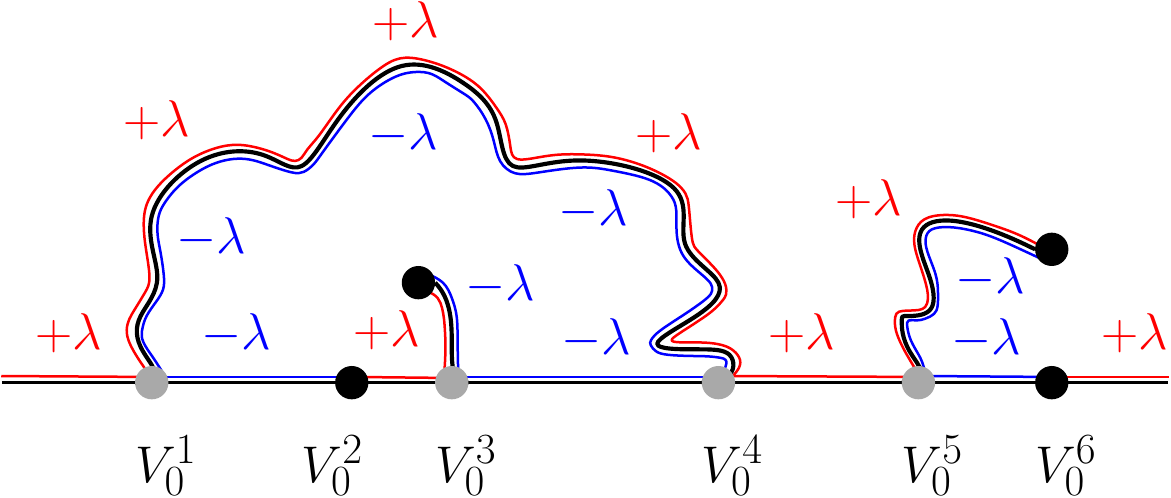}
\end{center}%
\caption{
\label{fig:GFF schematic}
The boundary conditions in the conditional law of Proposition~\ref{prop:GFF lvl lines}(ii) extend the original boundary condition in a natural ``cliff line'' manner.
}
\end{figure}

\begin{proposition}[GFF level lines]
\label{prop:GFF lvl lines}
Let $\hat{\Gamma}$ be the GFF in $\bbH$ with alternating boundary conditions, introduced in Section~\ref{subsec:GFF lvl lines}.
There exists a coupling of $\hat{\Gamma}$ and Loewner growth processes starting from the boundary points $V^1_0, \ldots, V^{2N}_0$, such that 
\begin{itemize}
\item[i)] the Loewner growth processes are generated by only $N$ disjoint simple chordal curves $ \gamma_1, \ldots, \gamma_{N}$ and their reversals,
 staying inside $\bbH$ except at the end points; and 
\item[ii)] any collection of 
initial segments $\eta_{\imath_1}, \ldots, \eta_{\imath_k}$ and full
 curves $ \gamma_{\jmath_1}, \ldots, \gamma_{\jmath_\ell} $ is a local set for $\hat{\Gamma}$, and the corresponding conditional law of $\hat{\Gamma}$ is a GFF in the reduced domain (independent GFFs on each connected component), with the boundary condition depicted in Figure~\ref{fig:GFF schematic}.
\end{itemize}
Furthermore,
\begin{itemize}
\item[iii)] there is only one coupling of $N$ random curves with endpoints $V^1_0, \ldots, V^{2N}_0$ and $\hat{\Gamma}$ that satisfies (i)--(ii);
\item[iv)] in this unique coupling the curves $\gamma_{1}, \ldots, \gamma_{N}$ are determined by $\hat{\Gamma}$; and
\item[v)] the law of the curves $ \gamma_1, \ldots, \gamma_{N}$ is the following: the marginal law of a single curve 
(or initial segment) from a given boundary point $V^0_j$
is the local multiple SLE of~\eqref{eq:NSLE SDE} with $\kappa = 4$ and the partition function~\eqref{eq:GFF ptt fcn} (such a curve always terminates at a boundary point of opposite index parity),
and given any collection of 
full curves, the conditional law of any remaining curve is the analogous local multiple SLE in the reduced domain, between the remaining boundary points.
\end{itemize}
\end{proposition}

For the context of the present note, first of all, that part~(i) above guarantees that the level lines indeed form some planar pairing between the points $V^1_0, \ldots, V^{2N}_0$. Secondly, we have the following (also in~\cite[Theorem~1.4]{PW}):

\begin{corollary}[GFF level lines are global multiple SLEs]
\label{cor:GFF lvl lines}
For any pairing $\alpha$ appearing for the GFF level lines with a positive probability\footnote{\textit{A posteriori}, by Theorem~\ref{thm:GFF lvl lines}, all planar pairings indeed appear with a positive probability.}, the level lines conditional on the pairing $\alpha$ are an $\alpha$-global multiple $\mathsf{SLE}(4)$, as defined in~\cite{BPW}.
\end{corollary}


The rest of this appendix constitutes the proofs of Proposition~\ref{prop:GFF lvl lines} and Corollary~\ref{cor:GFF lvl lines}.

\nopagebreak{

\begin{proof}[of Proposition~\ref{prop:GFF lvl lines}]
Roughly, the idea of the proof is to first consider curves sampled from the explicit distribution~(v) and prove (the natural analogues of)~properties (i), (ii) and (iv) for more and more general collections of such curves.
First, (i),~(ii) and~(iv) for a single initial segment as in~(v) can be found in~\cite{IG1}.\footnote{
A direct computation shows that the multiple $\SLE(4)$ of~\eqref{eq:NSLE SDE} with the partition function~\eqref{eq:GFF ptt fcn} indeed coincides with the SLE$(\kappa = 4, \rho)$ curves of~\cite{IG1}.
} Next, consider these statements for a single full curve: for part~(i), it is known that the curve obtained from increasing the initial segments given by~(v) indeed terminates at a boundary point of opposite index parity: for an argument suitable in our context, the partition functions $\calZ_\alpha$ describe global multiple SLE curves with this property~\cite{PW}, and by~\eqref{eq:GFF cvx space} and Lemma~\ref{lem:ptt fcn conv space ppty}, the initial segments from~(v) are just a convex combination of such measures. To show (ii), i.e., that such a full curve indeed is a local set with the asserted boundary conditions, one should follow~\cite[Proof of Proposition~5.8]{Werner-GFF} that addresses the case of two marked points.\footnote{%
To re-do the proof, one should replace the ``argument martingale'' of the ordinary SLE$(4)$ by $H(g_t(z); V^1_t, \ldots, V^{j-1}_t, W_t, V^{j+1}_t, \ldots, V^{2N}_t )$, where
$
H(z; x_1, \ldots, x_{2N}) := \Im \big( \log (z - x_1) -  \log (z - x_2) + \ldots -  \log (z - x_{2N}) \big).
$%
} 
For part (iv), the full curve is determined by $\hat{\Gamma}$ since all its initial segments were, and by part~(i), the full curve is the union of its increasingly long initial segments.

Next, construct random curves $\gamma_{1}, \ldots, \gamma_{N}$ (for which the statement above will soon turn out to be true) as follows: first, $\gamma_1$ is the above-constructed level line from $V^1_0$, terminating at an even-index boundary point. Then, in the domain reduced by $\gamma_1$, we construct $\gamma_2$ as a similar level-line from $V^3_0$, etc. 
Such random curves satisfy property (i) by construction. By the above cases and the conformal property of the GFF, the collection of full curves $\gamma_{1} \cup \ldots \cup \gamma_{N}$ is also a local set with the asserted boundary conditions (claim (ii)), and satisfies (iv). It is now that claim (iii) enters the proof; it (or more precisely, a suitable variant only referring to the local set property of full curves) is proven in~\cite[Proposition~5.15]{Werner-GFF}. Now, had we, for instance, first grown a curve from $V^3_0$ and only then the one from $V^1_0$, then we would get an \textit{a priori} different collection of random curves with the conclusions of this paragraph still true. However, due to property~(iii), the law of such two collections of random curves is actually the same. Formally, we just concluded the random curves $\gamma_{1}, \ldots, \gamma_{N}$ above indeed are described by~(v), and satisfy~(i)--(iv).

It remains to consider collections of both full curves and initial segments. Then, only~(ii) does not follow from the case of collections of full curves. For that matter, consider first only one initial segment $\eta_{\imath_1}$ and the full
 curves $ \gamma_{\jmath_1}, \ldots, \gamma_{\jmath_\ell} $. By~(v), we can choose to ``grow $\eta_{\imath_1}$ last'' and~(ii) then follows from the case of a single initial segment, applied in the reduced domain. Finally, multiple initial segments are handled by the union rule of local sets~\cite[Propositions~4.13 and~4.23]{Werner-GFF} applied to the union of $\eta_{\imath_i} \cup \gamma_{\jmath_1} \cup \ldots \cup \gamma_{\jmath_\ell} $ over $i$.
\end{proof}

\begin{proof}[of Corollary~\ref{cor:GFF lvl lines}]
Recall that the global multiple $\mathsf{SLE}(\kappa)$, $\kappa \leq 4$, in  $(\bbH; V_0^1,  \ldots, \newline V_0^{2N})$ with pairing $\alpha$ is a collection of random disjoint curves forming the pairing $\alpha$ between the boundary points, for which, given any $(N-1)$ curves, the conditional law of the remaining one is a chordal $\mathsf{SLE}(\kappa)$ between the remaining two boundary points in the domain reduced by the first $(N-1)$ curves~\cite{BPW}. 
Consider thus the $\alpha$-paired GFF level lines, and the conditional law of the $j$:th curve $\gamma_j$ given $\gamma_i$, $i \neq j$. Sample first $\gamma_i$, $i \neq j$, via Proposition~\ref{prop:GFF lvl lines}(v). Note that already these $(N-1)$ curves reveal the pairing formed by the level lines. Given that the pairing $\alpha$ is observed, by Proposition~\ref{prop:GFF lvl lines}(v) again, the conditional law of $\gamma_j$ given  $\gamma_i$, $i \neq j$ is the (unique) local multiple $\mathsf{SLE}(4)$ between remaining marked boundary points in the remaining sub-domain, which is well-known to coincide with the chordal $\mathsf{SLE}(4)$.
\end{proof}
}

\end{appendix}

\end{document}